\documentclass[12pt,a4paper]{amsart}

\usepackage{amssymb}

\usepackage{graphicx,amssymb,amsfonts,epsfig,amsthm,a4,amsmath,url}
\usepackage[latin1]{inputenc}

\newtheorem{thm}{Theorem}[section]

\newtheorem{lem}[thm]{Lemma}

\newtheorem{prop}[thm]{Proposition}

\theoremstyle{definition}

\newtheorem{rem}[thm]{Remark}

\numberwithin{equation}{section}

\newcommand{\Z}{\mathbf{Z}}

\newcommand{\C}{\mathbf{C}}
\newcommand{\Q}{\mathbf{Q}}

\newcommand{\GL}{\textnormal{GL}}
\newcommand{\PGL}{\textnormal{PGL}}

\newcommand{\Aut}{\mbox{Aut}}

\newcommand{\Cre}{\textnormal{Cr}}

\begin{document}
\title{Nonlinearity of some subgroups of the planar Cremona group}
\author{Yves Cornulier}%
\address{Laboratoire de Math\'ematiques\\
B\^atiment 425, Universit\'e Paris-Sud 11\\
91405 Orsay\\FRANCE}
\email{yves.cornulier@math.u-psud.fr}

\date{February 22, 2013}

\begin{abstract}We give some examples of non-nilpotent locally nilpotent, and hence nonlinear subgroups of the planar Cremona group.
\end{abstract}

\maketitle

\section{Introduction}

Let $K$ be a field. The planar {\em Cremona group} $\Cre_2(K)$ of $K$ is defined as the group of birational transformations of the 2-dimensional $K$-affine space. It can also be described as the group of $K$-automorphisms of the field of rational functions $K(t_1,t_2)$. More generally, one defines $\Cre_d(K)$.

We provide here two observations about the planar Cremona group. The first is an example of a non-linear finitely generated subgroup of $\Cre_2(\C)$. The existence of such a subgroup was known to some experts: for instance it follows from an unpublished construction of S.~Cantat (using superrigidity of lattices); our example has the additional feature of being 3-solvable. Its non-linearity follows from the fact it contains nilpotent subgroups of arbitrary large nilpotency length. We also show there that $\Cre_2(K)$ has no nontrivial linear representation over any field, extending a result of Cerveau and D\'eserti (all representations below are assumed finite-dimensional).

We end this short introduction by a few questions.

\begin{enumerate}
\item (Cantat) for $d\ge 2$, and any field $K$, is $\Cre_d(K)$ locally residually finite (i.e.\ is every finitely generated subgroup residually finite)? 

\item Does there exist a finitely generated subgroup of $\Aut(\C^2)$ with no faithful linear representation (see Remark \ref{autc2})?

\item Does there exist $d$ and $K$ and an infinite, finitely generated subgroup of $\Cre_d(K)$ such that every linear representation of $\Gamma$ over any field has a finite image

\end{enumerate}

\medskip

\noindent{\bf Acknowledgements.} I thank Serge Cantat and Julie Deserti for useful discussions.

\section{A nonlinear subgroup of the Cremona group}\label{nonlinear}

We provide in this section an example of a finitely generated subgroup of $\Cre_2(\C)$ that is not linear over any field. It is 3-solvable and actually lies in the Jonqui\`eres subgroup, that is, the group of birational transformations preserving the partition of $\C^2$ by horizontal lines. 

If $f\in K(X)$ and $g\in K(X)^\times$, define $\alpha_f,\mu_g\in\Cre_2(K)$ by
$$\alpha_f(x,y)=(x,y+f(x));\quad \mu_g(x,y)=(x,yg(x)).$$
We have 
$$\alpha_{f+f'}=\alpha_f\alpha_{f'};\quad\mu_{gg'}=\mu_g\mu_{g'};\quad\mu_g\alpha_f\mu_{g}^{-1}=\alpha_{fg}.$$
Also for $t\in K$, define $s_t\in\Cre_2(K)$ by $(x,y)=(x+t,y)$, so that
$$s_t\alpha_{f(X)}s_{t}^{-1}=\alpha_{f(X-t)};\quad s_t\mu_{g(X)}s_t^{-1}=\mu_{g(X-t)}.$$

Consider the subgroup $\Gamma_n$ of $\Cre_2(K)$ generated by $s_1$ and $\alpha_{X^n}$ ($n\ge 0$). 

\begin{lem}
The group $\Gamma_n$ is nilpotent of class at most $n+1$; moreover if $K$ has characteristic zero the nilpotency length of $\Gamma_n$ is exactly $n+1$, and $\Gamma_n$ is torsion-free.
\end{lem}
\begin{proof}
Consider the largest group $R_n$, generated by $s_1$ and by the abelian subgroup $A_n$ consisting of all $\alpha_{P}$, where $P$ ranges over polynomials of degree at most $n$. Then $A_n$ is normalized by $s_1$ and $[s_1,A_n]\subset A_{n-1}$ for all $n\ge 1$, while $A_0=\{1\}$. Therefore $R_n$ is nilpotent of class at most $n+1$, and therefore so is $\Gamma_n$. Conversely, the $n$-iterated group commutator $[s_1,[s_1,\cdots,[s_1,\alpha_{X^n}]\cdots]]$ is equal to $\alpha_{\Delta^nX^n}$, where $\Delta$ is the discrete differential operator $\Delta P(X)=-P(X)+P(X-1)$. So if $K$ has characteristic zero (or $p>n$) then $\Delta^nX^n\neq 0$ and $\Gamma_n$ is not $n$-nilpotent. In this case it is also clear that $R_n$ is torsion-free.
\end{proof}

Now assume that $K$ has characteristic zero and consider the group $G\subset\Cre_2(\Q)\subset\Cre_2(K)$ generated by $\{s_1,\alpha_1,\mu_X\}$.

\begin{prop}
The finitely generated group $G\subset\Cre_2(\Q)$ is solvable of length three; it is not linear over any field.
\end{prop}
\begin{proof}
From the conjugation relations above it is clear that the subgroup generated by $s_1$, all $\alpha_f$ and $\mu_g$, is solvable of length at most three. If we restrict to those $g$ of the form $\prod_{n\in\Z}(X-n)^{k_n}$ (where $(k_n)$ is finitely supported), we obtain a subgroup containing $\Gamma$, that is clearly torsion-free.

Since $\mu_X^n\alpha_1\mu_{X}^{-n}=\alpha_{X^n}$, we see that $G$ contains $\Gamma_n$ for all $n$, which is nilpotent of length exactly $n+1$. Therefore it has no linear representation over any field.

[Sketch of proof of the latter (well-known) result: in characteristic $p>0$, any torsion-free nilpotent subgroup is abelian, so this discards this case. Otherwise in characteristic zero, since any finite index subgroup of a torsion-free nilpotent group of nilpotency length $n+1$ still has nilpotency length $n+1$, the existence of a linear representation of $G$ into $\GL_d(\C)$ implies the existence of a Lie subalgebra of $\mathfrak{gl}_d(\C)$ of nilpotency length $n+1$ for all $n$; this necessarily implies $n+1\le d^2$, and since $n$ is unbounded this is a contradiction.]

The fact that $G$ is not 2-solvable (=metabelian) can be checked by hand, but also follows from the fact that every torsion-free finitely generated metabelian group is linear over a field of characteristic zero \cite{Rem}.
\end{proof}

We easily see $\Gamma_n\subset\Gamma_{n+1}$ for all $n$. Denoting $\Gamma_\infty=\bigcup\Gamma_n$, we see that $\Gamma_\infty$ is locally nilpotent (that is, all its finitely generated subgroups are nilpotent) and the above argument works for it. Since $\Gamma_\infty$ is contained in $\Aut(\C^2)$, we also get:

\begin{prop}
$\Gamma_\infty$, and hence $\Aut(\C^2)$ is not linear over any field.
\end{prop}

With little further effort, it actually follows from the same argument that $\Gamma_\infty$ (and hence $G$) is not linear over any finite product of fields. Better, it is not linear over any product of fields (and therefore over any reduced commutative ring). This now relies on the fact that nilpotent subgroups of $\GL_d(K)$ for fixed $d$, have nilpotency length bounded independently of the characteristic of the field $K$ (see \cite{FN}).


\begin{rem}\label{autc2}
It is unknown whether there exists a finitely generated subgroup of the group $\Aut(\C^2)$ of {\em polynomial automorphisms} of $\C^2$, that is not linear in characteristic zero. A construction in the same fashion does not work: indeed let $E$ be the group of elementary automorphisms, namely of the form $(x,y)\mapsto (\alpha x+P(y),\beta y+c)$ for $(\alpha,\beta,c,P)\in\C^*\times\C^*\times\C\times\C[X]$. Then, although $E$ is not linear (since by the argument above, it contains all $\Gamma_n$), every finitely generated subgroup of $E$ is linear over $\C$.

To see this, write $E$ as a semidirect product $(\C^*\times(\C^*\ltimes \C))\ltimes \C[X]$, where the action on $\C[X]$ is by $(\alpha,\beta,c)\cdot P(X)=\alpha P(\beta X+c)$. In particular, this action stabilizes the subgroup $\C_n[X]$ of polynomials of degree at most $n$. Therefore any finitely generated subgroup of $E$ is contained in the subgroup $(\C^*\times(\C^*\ltimes \C))\ltimes \C_n[X]$ for some $n\ge 1$. This is a (finite-dimensional) complex Lie group whose center is easily shown to be trivial, so its adjoint representation is a faithful complex linear representation. 
\end{rem}

A nice observation by Cerveau and D\'eserti \cite[Lemme~5.2]{CD} is that the Cremona group has no faithful linear representation in characteristic zero. Actually, an easy refinement of the same argument provides a stronger result. 

\begin{prop}\label{nolr}
If $K$ is an algebraically closed field, there is no nontrivial finite-dimensional linear representation of $\Cre_2(K)$ over any field.
\end{prop}

(Note that since the Cremona group is not simple by a recent difficult result of Cantat and Lamy \cite{CL}, the non-existence of a faithful representation does not formally imply the non-existence of a nontrivial representation.)

\begin{proof}[Proof of Proposition \ref{nolr}] In $\Cre_2(K)$, there is a natural copy of $G=(K^\times)^2\rtimes\Z$, where $\Z$ acts by the automorphism $\sigma(x,y)=(x,xy)$ of $(K^\times)^2$. Here, it corresponds, in affine coordinates, to the group of transformations of the form $$(x_1,x_2)\mapsto (\lambda_1x_1,x_1^n\lambda_2 x_2)\quad\text{for}\quad(\lambda_1,\lambda_2,n)\in (K^\times)^2\times\Z.$$ Consider an linear representation $\rho:G\to\GL_n(F)$, where $F$ is any field (here $G$ is viewed as a discrete group). If $p$ is a prime which is nonzero in $K$ and if $\omega_p\in K$ is a primitive $p$-root of unity, set $\alpha_p(x_1,x_2)=(\omega_px_1,\omega_px_2)$ and $\beta_p(x_1,x_2)=(x_1,\omega_px_2)$. Then $\sigma\alpha_p\sigma^{-1}\alpha_p^{-1}=\beta_p$ and commutes with both $\sigma$ and $\alpha_p$. An argument of Birkhoff \cite[Lemma~1]{Birk} shows that if $\rho(\alpha_p)\neq 1$ then $n>p$ (the short argument given in the proof of \cite[Lemme~5.2]{CD} for $F$ of characteristic zero works if it is assumed that $p$ is not the characteristic of $F$).

Picking $p$ to be greater than $n$ and the characteristics of $K$ and $F$, this shows that if we have an arbitrary representation $\pi:\Cre_2(K)\to\GL_n(F)$, the restriction of $\pi$ to $\PGL_3(K)$ is not faithful; since $\PGL_3(K)$ is simple, this implies that $\pi$ is trivial on $\PGL_3(K)$; since $\Cre_2(K)$ is generated by $\PGL_3(K)$ as a normal subgroup, this yields the conclusion.
\end{proof}


\begin{thebibliography}{KM98b}



\bibitem[Bi]{Birk} G. Birkhoff. Lie groups simply isomorphic with no linear group. Bull. Amer. Math. Soc. 42(12) (1936), 883--888.






\bibitem[CD]{CD} D. Cerveau, J. D\'eserti. Transformations birationnelles de petit degr\'e. Cours Sp\'ecialis\'es 19, Soc. Math. France, Paris, 2013.


\bibitem[CL]{CL} S. Cantat, S. Lamy. Normal subgroups of the Cremona group. Acta Math. 210 (2013), no. 1, 31--94.

\bibitem[FN]{FN} M. Frick, M.F. Newman.
Soluble linear groups. 
Bull. Austral. Math. Soc. 6 (1972), 31--44.


















\bibitem[Re]{Rem} V. Remeslennikov. Representation of finitely generated metabelian groups by matrices. Algebra i Logika 8 (1969), 72--75 (Russian); English translation in Algebra and Logic 8 (1969), 39--40. 












\end{thebibliography}
\end{document}